\documentclass[11pt]{article}

\usepackage{amsmath, amsfonts,amssymb, amsthm}
\usepackage{graphicx}
\usepackage{color}
\usepackage{epsf}
\usepackage{tikz}
\usepackage{subfig}
\usepackage{hyperref}
\hypersetup{
    colorlinks = true,
    linkcolor = {black},
    citecolor = {black}
}

\usepackage[left=1.25in, right=1.25in, top=1.2in, bottom=1.2in]{geometry}

\usepackage[sort,numbers]{natbib}
\usepackage{enumitem}
\setlist{itemsep=0pt, topsep=0pt}

\numberwithin{equation}{section}

\newcommand{\floor}[1]{\left\lfloor#1\right\rfloor}
\newcommand{\ceiling}[1]{\left\lceil#1\right\rceil}

\newcommand{\ext}{\mathrm{ext}}
\newcommand{\ex}{\mathrm{ex}}

\newcommand{\cH}{\mathcal{H}}

\newcommand{\tbf}[1]{\textbf{#1}}

\newtheorem{theorem}{Theorem}[section]
\newtheorem{corollary}[theorem]{Corollary}
\newtheorem{lemma}[theorem]{Lemma}

\newtheorem{claim}[theorem]{Claim}
\newtheorem{proposition}[theorem]{Proposition}
\newtheorem{observation}[theorem]{Observation}

\newtheorem{problem}[theorem]{Problem}

\newcommand{\db}[1]{{\color{red}#1}}

\newenvironment{proofclaim}[1][Proof of claim]{\begin{proof}[#1]}{\end{proof}}

\def\al#1{}
	\renewcommand{\al}[1]{\footnote{\textbf{AL: }#1}} 

\def\ld#1{}
	\renewcommand{\ld}[1]{\footnote{\textbf{LD: }#1}} 
	
\def\db#1{}
	\renewcommand{\db}[1]{\footnote{\textbf{DB: }#1}}

\title{A lower bound on the multicolor size-Ramsey numbers of paths in hypergraphs}
\author{Deepak Bal\thanks{Department of Mathematics, Montclair State University, Montclair, NJ {\tt deepak.bal@montclair.edu}.}  \and Louis DeBiasio\thanks{Department of Mathematics, Miami University, Oxford, OH. \texttt{debiasld@miamioh.edu}. Research supported in part by NSF grant DMS-1954170.} \and Allan Lo\thanks{School of Mathematics, University of Birmingham, Birmingham, B15 2TT, UK. \texttt{s.a.lo@bham.ac.uk} The research leading to these results was supported by EPSRC, grant no. EP/V002279/1 and EP/V048287/1. There are no additional data beyond that contained within the main manuscript.}}
\date{\today}

\begin{document}

\maketitle

\begin{abstract}
The $r$-color size-Ramsey number of a $k$-uniform hypergraph~$H$, denoted by $\hat{R}_r(H)$, is the minimum number of edges in a $k$-uniform hypergraph~$G$ such that for every $r$-coloring of the edges of~$G$ there exists a monochromatic copy of~$H$. 
In the case of $2$-uniform paths $P_n$, it is known that $\Omega(r^2n)=\hat{R}_r(P_n)=O((r^2\log r)n)$ with the best bounds essentially due to Krivelevich~\cite{K}.  In a recent breakthrough result, Letzter, Pokrovskiy, and Yepremyan~\cite{LPY} gave a linear upper bound on the $r$-color size-Ramsey number of 
the $k$-uniform tight path $P_{n}^{(k)}$; i.e.~$\hat{R}_r(P_{n}^{(k)})=O_{r,k}(n)$.  At about the same time, Winter~\cite{W1} gave the first non-trivial lower bounds on the 2-color size-Ramsey number of $P_{n}^{(k)}$ for $k\geq 3$; i.e.~$\hat{R}_2(P_{n}^{(3)})\geq \frac{8}{3}n-O(1)$ and $\hat{R}_2(P_{n}^{(k)})\geq \ceiling{\log_2(k+1)}n-O_k(1)$ for $k\geq 4$.

We consider the problem of giving a lower bound on the $r$-color size-Ramsey number of $P_{n}^{(k)}$ (for fixed $k$ and growing $r$).  
Our main result is that $\hat{R}_r(P_n^{(k)})=\Omega_k(r^kn)$ which generalizes the best known lower bound for graphs mentioned above.  One of the key elements of our proof  turns out to be an interesting result of its own.  We prove that $\hat{R}_r(P_{k+m}^{(k)})=\Theta_k(r^m)$ for all $1\leq m\leq k$; that is, we determine the correct order of magnitude of the $r$-color size-Ramsey number of every sufficiently short tight path.  

All of our results generalize to $\ell$-overlapping $k$-uniform paths $P_{n}^{(k, \ell)}$.  In particular we note that when $1\leq \ell\leq \frac{k}{2}$, we have $\Omega_k(r^{2}n)=\hat{R}_r(P_{n}^{(k, \ell)})=O((r^2\log r)n)$ which essentially matches the best known bounds for graphs mentioned above.  Additionally, in the case $k=3$, $\ell=2$, and $r=2$, we give a more precise estimate which implies $\hat{R}_2(P^{(3)}_{n})\geq \frac{28}{9}n-O(1)$, improving on the above-mentioned lower bound of Winter in the case $k=3$.
\end{abstract}

\section{Introduction}
Given hypergraphs $G$ and $H$ and a positive integer $r$, we write $G\to_r H$ to mean that in every $r$-coloring of the edges of $G$, there exists a monochromatic copy of $H$.  Given a $k$-uniform hypergraph $H$, the \emph{$r$-color size-Ramsey number of~$H$}, denoted by~$\hat{R}_r(H)$, is the minimum number of edges in a $k$-uniform hypergraph $G$ such that $G\to_r H$. The \emph{$r$-color Ramsey number of~$H$}, denoted by~$R_r(H)$, is the minimum number of vertices in a $k$-uniform hypergraph $G$ such that $G\to_r H$.  If $r=2$, then we drop the subscript.

For all integers $0\leq \ell\leq k-1$ and positive integer $m$, a $k$-uniform $\ell$-overlapping path (or a $(k,\ell)$-path for short) with $m$ edges is a $k$-uniform graph on vertex set $\{v_1, \dots, v_{k+(m-1)(k-\ell)}\}$ with edges $\{v_{(i-1)(k-\ell)+1}, \dots, v_{(i-1)(k-\ell)+k}\}$ for all $i\in [m]$ (note that the case $\ell=0$ corresponds to a matching). 
Note that any pair of consecutive edges has exactly $\ell$ vertices in common.
A $(k,\ell)$-path with $n$ vertices, denoted by~$P_{n}^{(k, \ell)}$, has $\frac{n-\ell}{k-\ell}$ edges (so whenever we write $P_{n}^{(k, \ell)}$, we are implicitly assuming that $k-\ell$ divides $n-\ell$).  
We sometimes write $P_{\ell + m(k-\ell)}^{(k, \ell)}$ to emphasize that the $(k,\ell)$-path has $m$ edges. 
If $\ell = 0, 1, k-1$, then $P_{n}^{(k,\ell)}$ is a {$k$-uniform matching}, a {$k$-uniform loose path} and a {$k$-uniform tight path}, respectively. 
We write $P^{(k)}_{n}$ and $P_n$ for $P^{(k,k-1)}_{n}$ and $P^{(2)}_{n}$, respectively.

For graphs, Beck~\cite{B} proved $\hat{R}(P_n)=O(n)$ and the best known bounds~\cite{BD,DP1} are $(3.75-o(1))n\leq \hat{R}(P_n)<74n$.  For the $r$-color version, it is known~\cite{DP2,DP3,K,BD} that 
\begin{align}
\Omega(r^2n)=\hat{R}_r(P_n)=O((r^2\log r) n). \label{eqn:sizepath}
\end{align}
After a number of partial results, e.g.~\cite{DLMR,LW,HKLMP}, Letzter, Pokrovskiy, and Yepremyan~\cite{LPY} recently proved that for all $r\geq 2$ and $1\leq \ell\leq k-1$, $\hat{R}_r(P_{n}^{(k, \ell)})=O_{r,k}(n)$ (their result covers more than just the case of paths, but for simplicity we don't state their general result here).
At around the same time, Winter provided the first non-trivial lower bounds on the size-Ramsey number of $P_{n}^{(k,\ell)}$ for 2 colors.
\begin{theorem}[Winter~\cite{W1,W2}] For all integers $n\geq k\geq 2$,
\begin{enumerate}
\item $\hat{R}(P^{(3)}_{n})\geq \frac{8}{3}n-\frac{28}{3}$.
\item $\hat{R}(P^{(k)}_{n})\geq \ceiling{\log_2(k+1)}n-2k^2$.
\item For all integers $\frac{2k}{3}< \ell\leq k-1$, $\hat{R}(P_{n}^{(k,\ell)})\geq \frac{\ceiling{\log_2\left(\frac{2k-\ell}{k-\ell}\right)}}{k-\ell}n-5k^2$.
\end{enumerate}
\end{theorem}

In this paper, we prove a general lower bound on the $r$-color size-Ramsey number of  $P_{n}^{(k,\ell)}$.  Our lower bound depends on both the $r$-color Ramsey number of the path~$P_{n}^{(k, \ell)}$ (for which the order of magnitude $R_r(P_{n}^{(k,\ell)})=\Theta_k(rn)$ is already known; see \eqref{eqn:bound_on_paths} below) and the $r$-color size-Ramsey number of the ``short'' path~$P_{\ell -1 + \floor{\frac{k}{k-\ell}}(k-\ell)}^{(k-1,\ell-1)}$ (for which essentially nothing was known).  See the first paragraph of Section~\ref{sec:link} for an explanation of the significance of $P_{\ell -1 + \floor{\frac{k}{k-\ell}}(k-\ell)}^{(k-1,\ell-1)}$.
 
\begin{theorem}\label{thm:main1ell}
Let $r\geq 2$ and $1\leq \ell\leq k-1$ be integers, and set $ q := \ell-1+\floor{\frac{k}{k-\ell}}(k-\ell)$ and $c_0:=c_0(r,k,\ell)=2k^2\hat{R}_r(P_{q}^{(k-1,\ell-1)})$.   For all integers $n\geq c_0$, we have
\begin{align*}
	\hat{R}_r(P_{n}^{(k,\ell)})\geq \frac1{k} 
	\hat{R}_r(P_{q}^{(k-1,\ell-1)}) R_r(P_{n-c_0}^{(k,\ell)}).
\end{align*} 
In particular, if $\ell=k-1$, then
$\hat{R}_r(P^{(k)}_{n})\geq \frac1{k} {\hat{R}_r(P_{2k-2}^{(k-1)}) R_r(P_{n-c_0}^{(k)}}).$
\end{theorem}

As mentioned above, it is known that $R_r(P_{n}^{(k, \ell)})=\Theta_{k}(rn)$.  More specifically, for all integers $r,k\geq 2$ and sufficiently large $n$,
\begin{equation}
	\label{eqn:bound_on_paths}
	\left( 1+ \frac{r-1}{{(k-\ell)\ceiling{\frac{k}{k-\ell}}}} \right) n-O_r(1)  < R_r(P_{n}^{(k,\ell)})\leq R_r(P^{(k)}_{n})< (1+o(1))rn,
\end{equation}
where the upper bound follows from a result of Allen, B{\"o}ttcher, Cooley, Mycroft~\cite{ABCM}
 and the lower bound is due to Proposition~\ref{prop:ramlower_ell} (there are more precise results for some specific values of $k$ and $r$, but we defer this discussion to Section~\ref{sec:ramseyturan}).  As a result, in order to get an explicit lower bound from Theorem~\ref{thm:main1ell}, it remains to get a lower bound on $\hat{R}_r(P_{q}^{(k-1,\ell-1)})$, where $q := \ell-1+\floor{\frac{k}{k-\ell}}(k-\ell)$.  In fact, we are able to determine $\hat{R}_r(P_{q}^{(k-1,\ell-1)})$ exactly up to a constant factor depending on~$k$; more specifically, we show $\hat{R}_r(P_{q}^{(k-1,\ell-1)})=\Theta_k(r^{\floor{\frac{k}{k-\ell}}-1})$ (or in the case of tight paths, $\hat{R}_r(P_{2k-2}^{(k-1)})=\Theta_k(r^{k-1})$).  We obtain this result as a consequence of the following more general theorem.

\begin{theorem}\label{thm:main2ell}
For all integers $1\leq \ell\leq k-1$ and $1\leq m\leq \floor{\frac{k}{k-\ell}}$, we have
   \[\hat{R}_r(P^{(k,\ell)}_{\ell+(m+1)(k-\ell)})=\Theta_{k}(r^{m}).\]
In particular, when $\ell=k-1$, this says that for all $1\leq m\leq k$, $\hat{R}_r(P^{(k)}_{k+m})=\Theta_{k}(r^{m})$.
\end{theorem}
In Section~\ref{sec:conclusion}, we have a further discussion about the hidden constants in Theorem~\ref{thm:main2ell}.

From \eqref{eqn:bound_on_paths} and Theorems~\ref{thm:main1ell} and~\ref{thm:main2ell} (with the appropriate reindexing -- see the proof in Section~\ref{sec:p4}), we get the following corollary.

\begin{corollary}\label{cor:ell}
For all integers $r, k \ge 2$ and $1\leq \ell\leq k-1$, $\hat{R}_r(P_{n}^{(k,\ell)})=\Omega_k(r^{\floor{\frac{k}{k-\ell}}} n)$.  
In particular, if $\ell=k-1$, then $\hat{R}_r(P_{n}^{(k)})=\Omega_k(r^{k} n)$.
\end{corollary}

In the case $k=2$, we are able to refine the bound in Theorem~\ref{thm:main2ell} by determining the $r$-color size-Ramsey number of $P_4$ for all $r\geq 2$ within a factor of 4.  

\begin{theorem}\label{thm:P_{k+2}} For all integers $r\geq 2$, 
$\frac{r^2}{2}< \hat{R}_r(P_{4})\leq (r+1)(2r+1)$.  
\end{theorem}

Additionally we note that for $k=r=2$, Harary and Miller~\cite{HM} proved that
\begin{equation}\label{P4}
\hat{R}(P_4)=7.
\end{equation}
So as a consequence of Theorem~\ref{thm:P_{k+2}} and \eqref{P4}, we get a refinement of Corollary~\ref{cor:ell} for tight paths when $k=3$.  

\begin{corollary}\label{cor:k3}~
$$
	\hat{R}_r(P^{(3)}_{n})
\ge
	\begin{cases}
	\frac{28 n }{9} - 390 & \text{if $r =2$,}\\
	\frac{r^2(r+2)}{12}n-O(r^5) & \text{if $r \ge 3$}.
\end{cases}
$$
\end{corollary}

Finally, when $1\le \ell \le k/2$, we slightly improve Corollary~\ref{cor:ell}.  We show that $\hat{R}_r(P_{n}^{(k,\ell)})=\Omega_k(r^{\ceiling{\frac{k}{k-\ell}}} n)=\Omega_k(r^{2} n)$, which matches, up to a $\log r$ factor, the best known upper bound (Observation~\ref{obs:ell} implies that $\hat{R}_r(P_{n}^{(k,\ell)})\leq \hat{R}_r(P_{n})$ and as mentioned above, Krivelevich~\cite{K} proved that $\hat{R}_r(P_{n})=O((r^2\log r)n)$).

\begin{theorem}\label{thm:loose-lb}
For all integers $r, k \ge 2$, $1\leq \ell \le k/2$, and $n > 12 r^2(k-\ell)+\ell$, we have $\hat{R}_r(P_{n}^{(k,\ell)})=\Omega_k(r^{2} n)$.
\end{theorem}

\subsection{Notation}

We sometimes write \emph{$k$-graph} to mean $k$-uniform hypergraph and \emph{$k$-path} to mean $k$-uniform tight path.  
Given a $k$-graph~$H$ and a vertex $v\in V(H)$, the \emph{link graph of $v$ in~$H$}, denoted by~$H_v$, is the $(k-1)$-graph on $V(H)$ such that $e \in \binom{V(H)}{k-1}$ is an edge of~$H_v$ if and only if $e \cup \{v\}$ is an edge of~$H$. 
For a subset $S \subseteq V(H)$, we write $H \setminus S$ to be the sub-$k$-graph of~$H$ obtained by deleting vertices of~$S$. 
We sometimes also identify~$H$ with its edge set. 

We write $f(r,k,d)=O_{k,d}(g(r,k,d))$ to mean that there exists a constant $C_{k,d}$, possibly depending on $k$ and~$d$, such that $f(r,k,d)\leq C_{k,d}\cdot g(r,k,d)$ for sufficiently large~$r$.  Likewise for $f(r,k,d)=\Omega_{k,d}(g(r,k,d))$ and $f(r,k,d)=\Theta_{k,d}(g(r,k))$.

For integers $a$ and $b$, we write $[b]$ for $\{1,\dots,b\}$ and $[a,b]$ for $\{a,a+1, \dots, b\}$.

\subsection{Organization of Paper} 
In Section~\ref{sec:ramseyturan}, we give lower bounds on the Ramsey numbers of $(k,\ell)$-paths and upper bounds on the size Ramsey numbers of short tight paths.  In Section~\ref{sec:link} we prove Theorem~\ref{thm:main1ell}, and in Section~\ref{sec:sr-short} we prove Theorem~\ref{thm:main2ell}. Section~\ref{sec:p4} contains proofs of the remaining results discussed above.  Finally, we discuss some open problems in Section~\ref{sec:conclusion}.

\section{Relationship between size-Ramsey numbers, Ramsey numbers, and Tur\'an numbers of paths}\label{sec:ramseyturan}

While there are many papers which give exact results~\cite{GG, GRSS}, asymptotic results~\cite{HLPRRSS1, HLPRRSS2, LP, GSS}, and bounds~\cite{KSu} on $R_r(P_{n}^{(k,\ell)})$ for some particular values of $r,k,\ell$, the general problem of (asymptotically) determining $R_r(P_{n}^{(k,\ell)})$ is completely wide open. 
For our purposes, we will only need a lower bound on $R(P^{(k,\ell)}_{n})$ which holds for all $n$, and an upper bound on $R(P^{(k)}_{n})$ which holds for $n=O(k)$.  We note that even if the exact value of $R_r(P^{(k)}_{n})$ was known for all $r,k,n\geq 2$, the only effect it would have on our results is slightly improving the hidden constants in Corollary~\ref{cor:ell}.

\subsection{A lower bound on the Ramsey number of $(k,\ell)$-paths}

Recall that $P^{(k,0)}_{km}$ is the $k$-uniform matching of size~$m$.
The following proposition gives a lower bound on $R_r(P_{n}^{(k,\ell)})$, which is based on the construction of Alon, Frankl, and Lov\'asz~\cite{AFL} giving a lower bound on $R_r(P^{(k,0)}_{km})$, and was essentially already observed in~\cite{HLPRRSS2}.

\begin{proposition}\label{prop:ramlower_ell}
For all integers $r,k\ge 2$, $m\geq 1$, $1 \le \ell \leq k-1$, and $n=\ell+m(k-\ell)$, we have 
\begin{align*}
R_r(P_{n}^{(k,\ell)})\geq 
(r-1) \ceiling{\frac{n-\ell}{(k-\ell)\ceiling{\frac{k}{k-\ell}}}} + n -r +1
>
\left( 1+ \frac{r-1}{{(k-\ell)\ceiling{\frac{k}{k-\ell}}}}\right)n - 2(r-1).
\end{align*}
In particular, when $k-\ell$ divides $k$, we have $R_r(P_{n}^{(k,\ell)})>\frac{r-1+k}{k} n -2(r-1)$ and when $1\leq \ell\leq k/2$, we have $R_r(P_{n}^{(k,\ell)})>\frac{r-1+2(k-\ell)}{2(k-\ell)} n -2(r-1)$.
\end{proposition}

\begin{proof}
Let $m' = \ceiling{ {m}/\ceiling{\frac{k}{k-\ell}}}$ and note that $P_{n}^{(k,\ell)}$ contains a matching of size~$m'$.  

Let $A_1, \dots, A_{r-1}$ be disjoint vertex sets each of order~$m'-1$ and one vertex set~$A_r$ of order~$n -1$.
Then take a complete $k$-graph on $A_1\cup \dots \cup A_r$ and color each edge~$e$ with the smallest index~$i$ such that $e\cap A_i\neq \emptyset$.  Note that for all $i\in [r-1]$, there is no matching of color~$i$ of order~$m'$ and thus there is no $(k,\ell)$-path of color~$i$ of order~$n$.
Also, all edges of color~$r$ are contained in~$A_r$ which is of order $n-1$, so there is no $(k,\ell)$-path of color~$r$ of order~$n$.
So we have $R_r(P^{(k,\ell)}_{n})> (r-1)(m'-1) + n-1 = (r-1)m'+n -r,$ as desired.
\end{proof}

We also note that in the case of graphs it is known~\cite{YYXB} that $R_r(P_n)\geq (r-1)(n-1)$ for all $r\geq 3$, which is better than the bound above when $k=2$ and $r\geq 4$.

We now consider the case when $r$ is sufficiently large compared to~$k$ and~$n$, in which case we can improve the lower bound by a factor of~$k$ for infinitely many $r$.  
Let $K_{n}^{(k)}$ be the complete $k$-graph on $n$ vertices. 
A ${K}_{n}^{(k)}$-decomposition of~$K_{N}^{(k)}$ is a collection $\mathcal{K}_{n}^{(k)}$ of edge-disjoint~$K_{n}^{(k)}$ such that every edge of~$K_{N}^{(k)}$ is contained in exactly one copy of~$K_{n}^{(k)}$.
Keevash~\cite{Ke} proved that a $K_{n}^{(k)}$-decomposition of~$K_{N}^{(k)}$ exists for sufficiently large $N$ subject to some necessary divisibility conditions.
Glock, K\"uhn, Lo, and Osthus~\cite{GKLO} gave an alternative proof\footnote{Note that the results in~\cite{Ke, GKLO} are phrased in terms of designs, so the reader should note that a $K_{n}^{(k)}$-decomposition of~$K_{N}^{(k)}$ is equivalent to a $k$-$(N,n,1)$ design.}.

\begin{theorem}[Keevash~\cite{Ke}, Glock, K\"uhn, Lo, and Osthus~\cite{GKLO}]\label{thm:design}
For all integers $1 \le k \le n$, there exists $N_0 = N_0(n,k)$ such that there is a ${K}_{n}^{(k)}$-decomposition of~$K_{N}^{(k)}$ for all integers $ N > N_0$ with $\binom{n-i}{k-i} | \binom{N-i}{k-i}$ for all $0 \le i \le k-1$.
\end{theorem}

We use Theorem~\ref{thm:design} to extend a result of Axenovich, Gy\'arf\'as, Liu, and Mubayi~\cite[Theorem 11]{AGLM} who proved that $R_r(P^{(3)}_{5})\geq 2r(1-o(1))$ for all $r\geq 2$.  As in~\cite{AGLM}, we use the following result of Pippenger and Spencer~\cite{PS} (stated here only for regular hypergraphs).

\begin{theorem}[Pippenger and Spencer~\cite{PS}]\label{thm:PS}
Every $D$-regular, $m$-uniform hypergraph having the property that every pair of vertices is contained together in $o(D)$ edges can be decomposed into $(1+o(1))D$ many matchings.
\end{theorem}

\begin{proposition}\label{prop:manycolour}
For all integers $n > k \ge 2$ there exists an $r_0=r_0(n,k)$ such that for infinitely many integers $r\geq r_0$, $R_r(P^{(k)}_{n})\geq r(n-k)( 1 - o(1) )$.
\end{proposition}

\begin{proof}
Let $N_0 = N_0(n-1,k)$ be given by~Theorem~\ref{thm:design}.
Let $N \ge N_0$ be an integer such that $\binom{n-1-i}{k-i} | \binom{N-i}{k-i}$ for all $0 \le i \le k-1$.
So there exists a ${K}_{n-1}^{(k)}$-decomposition~$\mathcal{K}_{n-1}^{(k)}$ of~$K_{N}^{(k)}$.
Let $V = V( K_{N}^{(k)} )$.
Define an $\binom{n-1}{k-1}$-graph~$H$ with $V(H) = \binom{V}{k-1}$ and $E(H) = \{ \binom{ V(K) }{k-1}: K\in \mathcal{K}_{n-1}^{(k)} \}$.
The degree of every~$S \in V(H)$ in $H$ is the number of $K_{n-1}^{(k)} \in \mathcal{K}_{n-1}^{(k)}$ containing~$S$, which is $(N-(k-1))/((n-1)-(k-1))=(N-k+1)/(n-k)$.
Given any two distinct $S,S' \in V(H)$, note that $S \cup S'$ contains an edge of $K_{N}^{(k)}$, so there is at most one $K_{n-1}^{(k)} \in \mathcal{K}_{n-1}^{(k)}$ containing~$S \cup S'$. 
Hence every pair of vertices in $H$ is contained together in at most one edge.
Now Theorem~\ref{thm:PS} implies that $H$ can be decomposed into at most $(1+o(1)) (N-k+1)/(n-k) $ matchings, and we color each such matching with a distinct color. 
Let $r$ be the number of colors used, so $N = (1-o(1)) r(n-k)$. 
Each matching corresponds to a set of $K_{n-1}^{(k)}$ which pairwise intersect in at most $k-2$ vertices, so the largest tight component has order $n-1$ and thus there are no monochromatic copies of $P_{n}^{(k)}$.
\end{proof}

\subsection{An upper bound on the size-Ramsey number of short tight paths}

We will later need an upper bound on $\hat{R}_r(P_{\ell+m(k-\ell)}^{(k,\ell)})$ for all integers $1\leq \ell\leq k-1$ and $2 \le m \le \ceiling{\frac{k}{k-\ell}}+1$; however, we will prove in Lemma~\ref{lem:monoton_ell} that $\hat{R}_r(P_{\ell+m(k-\ell)}^{(k,\ell)})\leq \hat{R}_r(P_{2m-2}^{(m-1)})$. As a result, we only discuss tight paths in this subsection.  

For graphs, it is known that 
\begin{align*}
R_r(P_3)=\begin{cases} 
      r+1 & \text{if } r \text{ is even}\\
			r+2 & \text{if } r \text{ is odd}
   \end{cases}
\end{align*}		
which is equivalent to determining the edge chromatic number of complete graphs.
Bierbrauer~\cite{B} proved the following surprisingly difficult result
\begin{equation}
\label{thm:bier}
R_r(P_4)= \begin{cases} 
      6 & \text{if } r=3,\\
			2r+2 & \text{if } r\equiv 1 \bmod 3, \\
      2r+1 & \textrm{otherwise.} 
   \end{cases}
\end{equation}

For 3-uniform hypergraphs, Axenovich, Gy\'arf\'as, Liu, and Mubayi~\cite[Theorem 7]{AGLM} determine $R_r(P_4^{(3)})$ exactly for all $r\geq 2$ (it is either $r+1$, $r+2$ or $r+3$ depending on divisibility conditions on $r$).  They also prove~\cite[Theorem 11]{AGLM} that $2r(1-o(1))\leq R_r(P_5^{(3)})\leq 2r+3$ for all $r\geq 2$.

Given a $k$-graph $H$, let $\ex(n,H)$ be the maximum number of edges in a $k$-graph~$G$ on $n$ vertices such that $H\not\subseteq G$.  Note that if $G$ is a $k$-graph on $n$ vertices with more than $\ex(n,H)$ edges, then $H\subseteq G$.

Kalai conjectured (see~\cite{FF}) that for all integers $N\geq n\geq k\geq 2$, $\ex(N, P^{(k)}_{n})\leq \frac{n-k}{k}\binom{N}{k-1}$.  While Kalai's conjecture is still open, a straightforward bound of 
\begin{equation}\label{eq:GKL}
\ex(N, P^{(k)}_{n})\leq (n-1)\binom{N}{k-1} 
\end{equation}
was given explicitly by Gy\"ori, Katona, and Lemons~\cite[Theorem 1.11]{GKL}, and in more generality by F\"uredi and Jiang~\cite{FJ} (see~\cite{FJKMV2} for further discussion).   The best known bounds are due to F\"{u}redi, Jiang, Kostochka, Mubayi, and Verstra\"{e}te~\cite{FJKMV}, but as their result depends on the parity of $k$ and only improves on \eqref{eq:GKL} by about a factor of 2, we will use the simpler bound given in \eqref{eq:GKL}.
 

We will also use the following trivial observation
\begin{equation}\label{eq:trivial}
\hat{R}_r(H)\leq \binom{R_r(H)}{k}.
\end{equation}
Combining \eqref{eq:GKL} and \eqref{eq:trivial} we have the following corollary (which will be useful when $n=O(k)$).  

\begin{corollary}\label{cor:upper}
For all integers $n,r,k\geq 2$, $R(P_{n}^{(k)})\leq rkn$ and thus 
$\hat{R}_r(P_{n}^{(k)})\leq \binom{rkn}{k} \leq \left( enr \right)^k.$
\end{corollary}

\begin{proof}
We first show that $R(P_{n}^{(k)})\leq rkn$.  Note that $$\frac{1}{r}\binom{rkn}{k}=\frac{rkn-k+1}{rk}\binom{rkn}{k-1}>(n-1)\binom{rkn}{k-1},$$
so the majority color class in any $r$-coloring of $K^{(k)}_{rkn}$ contains a copy of $P_{n}^{(k)}$ by \eqref{eq:GKL}.  

Thus by \eqref{eq:trivial}, we have $\hat{R}_r(P_{n}^{(k)})\leq \binom{rkn}{k}\leq (ern)^k$ as desired.
\end{proof}

Note that if Kalai's conjecture is true, then we would have $\hat{R}_r(P_{n}^{(k)})\leq \binom{rn}{k}$.  (So for example, if $n\leq 2k$, we would have $\hat{R}_r(P_{n}^{(k)})\leq \binom{rn}{k}\leq \left(\frac{ern}{k}\right)^k\leq (2er)^k$.)

\section{Proof of Theorem~\ref{thm:main1ell}}\label{sec:link}
In this section  we prove Theorem~\ref{thm:main1ell} (which can be thought of as a generalization of~\cite[Lemma 2.4, Corollary 2.5]{BD}).
Given a $(k,\ell)$-path $P=v_1v_2\dots v_p$, let the \emph{exterior} of~$P$, denoted by~$P_{\ext}$, be the set consisting of the first $\ell$ vertices and the last $\ell$ vertices of~$P$; that is $P_{\ext}=\{v_1,\dots, v_{\ell}\}\cup\{v_{p-\ell+1}, \dots, v_p\}$, and let the \emph{interior} of~$P$, denoted by~$P_{\mathrm{int}}$, be~$V(P)\setminus P_{\ext}$ (so for example in the case $k=2$, the endpoints of the path form the exterior and the remaining vertices form the interior).  The key observation is that if $v\in P_{\mathrm{int}}$, then the link graph of $v$ in $P$ is either a $(k-1,\ell-1)$-path with $\floor{\frac{k}{k-\ell}}$ edges or a $(k-1,\ell-1)$-path with $\ceiling{\frac{k}{k-\ell}}$ edges (and this is not necessarily the case if $v\in P_{\ext}$). 

\begin{proof}[Proof of Theorem~\ref{thm:main1ell}]
Recall that $ q = \ell-1+\floor{\frac{k}{k-\ell}}(k-\ell)$, set $\hat{d} = \hat{R}_r \left( P_{q}^{(k-1,\ell-1)} \right)$, and set $n'=n-2k^2\hat{d}$.
Let $H=(V,E)$ be a $k$-graph with 
\begin{equation}\label{eq:E}
|E|< \frac{\hat{d}}{k}  R_r(P_{n'}^{(k,\ell)}).
\end{equation}
We will show that there is an $r$-coloring of $H$ with no monochromatic $P_{n}^{(k,\ell)}$.

Let $S=\{v\in V: d(v)<\hat{d}\}$.  Since 
\begin{align*}
\hat{d}\cdot |V\setminus S|\leq \sum_{v\in V}d(v)=k\cdot |E| < \hat{d} \cdot R_r(P_{n'}^{(k,\ell)}),
\end{align*}
we have $|V\setminus S| <R_{r}(P_{n'}^{(k,\ell)})$ and thus we can color~$H \setminus S$ with $r$ colors so there is no monochromatic $P_{n'}^{(k,\ell)}$ in~$H \setminus S$.

Let $E_S$ be the set of edges from $E$ which intersect~$S$ and let $H_S$ be the $k$-graph induced by edges in~$E_S$.
Note that for every vertex in the interior of a $(k,\ell)$-path, its link graph contains a $(k-1,\ell-1)$-path with $\floor{\frac{k}{k-\ell}}$ edges; i.e.~$P_{q}^{(k-1,\ell-1)}$.
Recall that every $v \in S$ has degree less than $\hat{d} =\hat{R}_r(P_{q}^{(k-1,\ell-1)})$ in~$H_S$.  
So if we can color the edges of~$H_S$ so that there is no monochromatic $P_{q}^{(k-1,\ell-1)}$ in the link graph of every $v \in S$, then no vertex of~$S$ will be in the interior of any monochromatic $(k,\ell)$-paths.
Thus $H$ would have no monochromatic $(k,\ell)$-path of order $n' +2\ell$ (as there are no monochromatic $(k,\ell)$-paths of order $n'$ in $H-S$ and there are $2\ell$ vertices in the exterior of a $(k,\ell)$-path).
However, we cannot color each link graph of $v \in S$ independently since it is possible for an edge to contain more than one vertex from~$S$.  Instead, we are able to provide an $r$-coloring having the property that if a vertex of $S$ is in the interior of a monochromatic $(k,\ell)$-path, then it must be within distance $k^2\hat{d}$ from an endpoint of $P$ (which explains the reason for the definition of $n'$).  

We begin by partitioning $S$ into sets $S_1, \dots, S_c$ such that for all $e\in E_S$ and all $i\in [c]$, $|e\cap S_i|\leq 1$ with $c$ as small as possible.  Since  $H_S$ is a $k$-uniform hypergraph with maximum degree at most $\hat{d}-1$, the usual greedy coloring algorithm (color the vertices one by one and note that each vertex is contained in an edge with at most $(k-1)(\hat{d}-1)$ vertices which have already been colored) implies that
\begin{align*}
c \le (k-1)(\hat{d}-1)+1 \le k\hat{d}.
\end{align*}

For all $i\in [c]$, let $E_i$ be the set of edges which are incident with a vertex in~$S_i$ but no vertices in~$S_1\cup \dots \cup S_{i-1}$ and let $H_i$ be the $k$-graph induced by edges in~$E_i$.  Note that $\{E_1, \dots, E_c\}$ is a partition of~$E_S$. For each $i\in [c]$, $E_i$ can be further partitioned based on which vertex in~$S_i$ they are adjacent to.  For all $i\in [c]$ and all $v\in S_i$, we color the edges of $E_i$ incident with $v$ so that there are no monochromatic copies of $P_{q}^{(k-1,\ell-1)}$ in the link graph of~$v$ in~$H_i$ (which is possible by the definition of $\hat{d}$ and the previous sentences).

We now show that $H$ does not contain a monochromatic $(k,\ell)$-path on $n$ vertices. 
Suppose $P=v_1v_2\dots v_p$ is a monochromatic $(k,\ell)$-path in~$H$.
The next claim shows that if $v_j \in S$, then $v_j$ is not far from $v_1$ or~$v_p$. 

\begin{claim} \label{clm:mainell}
If $v_{j} \in  S$, then $\min \{j , p - j+1 \} \le \ell+(k-1)c$.
\end{claim}

\begin{proofclaim}[Proof of Claim~\ref{clm:mainell}]
Let $j_0 = j$ and $i_0 \in [c]$ be such that $v_{j_0} \in S_{i_0}$. 
If $v_{j_0} \in P_{\mathrm{ext}}$, then $\min \{j , p - j+1 \} \le \ell$ and so we are done. 
If $v_{j_0} \in P_{\mathrm{int}}$, then we will show that there is a vertex~$v_{j_1} \in S_{i_1}$ such that $|j_1 - j_0| \le k-1$ and $i_1 <i_0$.
We then repeat this argument for~$v_{j_1}$ and note that this will stop within $c$ rounds. 

Formally, we will show that there exists $b \in [c]$,  $j_0, j_1,\ldots, j_{b} \in [p]$, and $i_0 > i_1 >\dots > i_{b} \ge 1$ such that $v_{j_b} \in P_{\mathrm{ext}}$ and, for all $a \in [b]$, $|j_{a} - j_{a-1} | \le k-1$.
Since $v_{j_b} \in P_{\mathrm{ext}}$, we have $j_b \in [\ell] \cup [p-\ell+1, p]$. 
If $j_b \in [\ell]$, then 
we have 
\begin{align*}
	j_0 = j_{b} + \sum_{a \in [b]} \left( j_{a-1} - j_{a}  \right) \le j_{b} + \sum_{a \in [b]} |j_{a-1} - j_{a}| \le \ell + c(k-1)
\end{align*}
and a similar calculation shows that if $j_b \in  [p-\ell+1, p]$, then $ p - j+1  \le \ell+(k-1)c$. 
This implies the claim. 

Suppose that we have found $j_a$ and $i_a$ for some $a\geq 0$. 
If $j_a \in [\ell] \cup [p-\ell+1, p]$, then we are done. 
Suppose that $j_a \notin  [\ell] \cup [p-\ell+1, p]$, that is, $v_{j_a} \in P_{\mathrm{int}}$. 
Note that the link graph of $v_{j_a}$ in~$P$ contains $P_{q}^{(k-1,\ell-1)}$.
Thus by the coloring of the edges in~$E_S$, there must exist $j_{a+1} \in [j_a-k+1, j_a +k-1]$ and $i_{a+1} <i_{a}$ such that $v_{j_{a+1}}\in S_{i_{a+1}}$.
\end{proofclaim}

By Claim~\ref{clm:mainell}, we know that if we remove $\ell + c(k-1)$ vertices from each end of~$P$, then the resulting $(k,\ell)$-path lies in~$H \setminus S$, which has order at most $n'-1$ by the coloring of~$H \setminus S$. 
Therefore, we have $p \le 2 (\ell + c(k-1)) + n'-1<2k^2\hat{d}+n'=n$ as required. 
\end{proof}

\section{Size-Ramsey numbers of short $(k,\ell)$-paths}\label{sec:sr-short}
\subsection{Upper bound}
Dudek, La Fleur, Mubayi, and R\"odl~\cite{DLMR} observed that for all $1\leq \ell\leq \frac{k}{2}$, $\hat{R}(P_{n}^{(k, \ell)})\leq \hat{R}(P_{n})$ and Han, Kohayakawa, Letzter, Mota, and Parczyk~\cite{HKLMP} observed that if $3$ divides $k$, then 
$\hat{R}(P_{n}^{(k, 2k/3)})\leq \hat{R}(P_{n}^{(3)})$. 
 We begin with a generalization of these observations (the first part was already mentioned by Winter~\cite{W2} in the case $r=2$).

\begin{observation}\label{obs:ell}Let $n,r,k\geq 2$ be integers. 
\begin{enumerate}
\item For all $1\leq \ell\leq \frac{k}{2}$, $\hat{R}_r(P_{n}^{(k, \ell)})  \leq \hat{R}_r(P_{\frac{n-\ell}{k-\ell}+1})$.
\item For all $n\geq k>\ell \ge 1$ and $d\geq 1$ such that $d$ divides $n$, $k$, and $\ell$, $\hat{R}_r(P_{n}^{(k, \ell)})\leq \hat{R}_r(P_{n/d}^{(k/d,\ell/d)})$.
\item For $1\leq \ell\leq k-1$, and $1 \le m <  \frac{k}{k-\ell}$, 
 $ \hat{R}_r(P_{\ell+(m+1)(k-\ell)}^{(k,\ell)} ) \le \hat{R}_r(P_{\ell-1+(m+1)(k-\ell)}^{(k-1,\ell-1)}) $.
\end{enumerate}
\end{observation}

\begin{proof}
For (i), let $m = \frac{n-\ell}{k-\ell}$, so that $P_{n}^{(k, \ell)}$ has $m$ edges.
Begin with a graph~$G$ such that $G\to_r P_{m+1}$.
Now replace each vertex with a set of $\ell$ vertices and add an additional $k-2\ell$ unique vertices to each edge to get a $k$-graph~$H$ with the same number of edges as $G$ such that $H\to_r P_{n}^{(k, \ell)}$ (any $r$-coloring of $H$ corresponds to an $r$-coloring of $G$, and any monochromatic $P_{m+1}$ in $G$ corresponds to a monochromatic $P_{n}^{(k, \ell)}$ in $H$).

For (ii), begin with a $(k/d)$-graph~$G$ such that $G\to_r P_{n/d}^{(k/d,\ell/d)}$.  Now replace each vertex with a set of $d$ vertices to get a $k$-graph~$H$ with the same number of edges as $G$ such that $H\to_r P_{n}^{(k, \ell)}$.

For (iii), note that $P_{\ell+(m+1)(k-\ell)}^{(k,\ell)}$ contains a vertex~$v$ such that its link graph is precisely~$P_{\ell-1+(m+1)(k-\ell)}^{(k-1,\ell-1)}$.
Begin with a $(k-1)$-graph~$G$ such that $G\to_r P_{\ell-1+(m+1)(k-\ell)}^{(k-1,\ell-1)}$. 
Now let $y$ be a new vertex and add $y$ to each edge of~$G$ to get a $k$-graph $H$ with the same number of edges as~$G$ such that $H \to_r P_{\ell+(m+1)(k-\ell)}^{(k,\ell)}$.
\end{proof}

The following lemma provides an upper bound on the $r$-color size-Ramsey number of every sufficiently short $(k, \ell)$-path.

\begin{lemma}\label{lem:monoton_ell}
For all integers $r\geq 2$, $1\leq \ell\leq k-1$, and $1\leq m \leq \floor{\frac{k}{k-\ell}}$, 
\begin{align*}
	\hat{R}_r(P_{\ell+(m+1)(k-\ell)}^{(k,\ell)}) 
	\leq
		\hat{R}_r(P_{2m}^{(m)}) \leq (2emr)^m.
\end{align*} 
In particular (when $m=1$), we have $\hat{R}_r(P^{(k,\ell)}_{2k-\ell})=r+1$.
\end{lemma}

\begin{proof}
First suppose that $k-\ell$ divides $k$ and $m=\frac{k}{k-\ell}$.  In this case we have $\frac{\ell}{k-\ell}=\frac{k}{k-\ell}-1=m-1$ and $\frac{\ell+(m+1)(k-\ell)}{k-\ell}=2m$.  So by Observation~\ref{obs:ell}(ii) and Corollary~\ref{cor:upper}, we have $\hat{R}_r(P_{\ell+(m+1)(k-\ell)}^{(k,\ell)}) \leq \hat{R}_r(P_{2m}^{(m)})\leq (2emr)^m$.

Next suppose that $m < \frac{k}{k-\ell}$.  In this case, there is a (non-empty) set of $k-m(k-\ell)$ vertices which are contained in every edge of~$P_{\ell+(m+1)(k-\ell)}^{(k,\ell)}$.  
Take an $(m(k-\ell))$-graph $H$ with $\hat{R}_r(P_{2m(k-\ell)}^{(m(k-\ell),(m-1)(k-\ell))})$ edges such that $H\to_r \hat{R}_r(P_{2m(k-\ell)}^{(m(k-\ell),(m-1)(k-\ell))})$ and create a $k$-graph $H'$ having the same number of edges as $H$ by adding a set $U$ of $k-m(k-\ell)$ vertices which are contained in every edge of~$H$.  Now in every $r$-coloring of the edges of $H'$ (and by extension~$H$), there is a monochromatic copy of~$P^{(m(k-\ell),(m-1)(k-\ell))}_{2m(k-\ell)}$ in~$H$ and thus a monochromatic copy of~$P_{\ell+(m+1)(k-\ell)}^{(k,\ell)}$ in~$H'$.  Now by Observation~\ref{obs:ell}(ii) and Corollary~\ref{cor:upper}, we have $$\hat{R}_r(P_{2m(k-\ell)}^{(m(k-\ell),(m-1)(k-\ell))}) \leq \hat{R}_r(P_{2m}^{(m)}) \leq (2emr)^m.$$

Note that when $m=1$ we get the more precise bound because we trivially have $\hat{R}_r(P_{2m}^{(m)})=r+1$.
\end{proof}

\subsection{Lower bound}

The following lemma will be applied with $\floor{\frac{r-1}{2f(k,\ell,m)}}$ in place of $r$ to give us a lower bound on $\hat{R}_r(P_{\ell+(m+1)(k-\ell)}^{(k,\ell)})$ for all $1\leq m\leq \floor{\frac{k}{k-\ell}}$.  

\begin{lemma} \label{lemma:p2k_ell_d}
Let $r,k\geq 2$, $1\leq \ell\leq k-1$, and $1\leq m\leq \floor{\frac{k}{k-\ell}}$ be integers.  For all $j\in [m+1]$, set $t_j=  k-(j-1)(k-\ell)$ and
set
\begin{align*}
	f (k, \ell, m) : = \binom{k}{\ell} \sum_{j \in [m]} \binom{\ell}{t_{j+1}} 
	\text{ and }
	g(k,\ell,m) := \prod_{j \in [m-1]} \binom{t_j}{t_{j+1}}.
\end{align*}
If $H$ is a $k$-graph with $|H| \leq \frac{r^{m}}{g(k,\ell,m)}$, then there exists a $(2r\cdot f(k,\ell,m)+1)$-edge-coloring of~$H$ without a monochromatic~$P^{(k,\ell)}_{\ell+(m+1)(k-\ell)}$.
\end{lemma}

Note that if $\ell=k-1$, then we have $g(k,k-1, m) =k(k-1)\cdots(k-(m-2))$ and $f(k,k-1,m)=k\sum_{i=0}^{m-1}\binom{k-1}{i}$.

Before beginning the proof, we give a high level overview.  We are given a $k$-uniform hypergraph $H$ with a sufficiently small number of edges which we must color in such a way that there is no monochromatic $P_{\ell+(m+1)(k-\ell)}^{(k,\ell)}$.  In order to produce such a coloring, we actually prove something stronger.  We say that a pair of edges $(e,f)\in E(H)\times E(H)$ is ``dangerous'' if $|e\cap f|=\ell$ and they satisfy an additional special property (see \eqref{eqn:property} below).  As we will prove, every copy of $P_{\ell+(m+1)(k-\ell)}^{(k,\ell)}$ contains a dangerous pair of edges.  Furthermore, the definition of dangerous will imply that for every edge $e\in E(H)$, there are only a bounded number of $f\in E(H)$ for which $(e,f)$ is a dangerous pair.  This allows us to greedily color the edges of $H$ so that no dangerous pair of edges receives the same color and therefore there is no monochromatic copy of $P_{\ell+(m+1)(k-\ell)}^{(k,\ell)}$.  Now we proceed to the formal proof.

\begin{proof}[Proof of Lemma~\ref{lemma:p2k_ell_d}]
Set $B_{1}=H$.
For all $j\in [2,m+1]$, let $B_j$ be the $t_j$-graph on $V(H)$ such that
\begin{align*}
	B_{j} & = \left\{ S \in \binom{V}{t_j} : \deg_{B_{j-1}}(S) > r \right\}.
\end{align*}
In words, for all $j\in [2,m+1]$, $B_{j}$ is a $t_j$-graph, whose edges are those with degree ``too large'' in~$B_{j-1}$.  We first establish the following claim.

\begin{claim}\label{clm:i0}
$B_{m+1}=\emptyset$.
\end{claim}

Note that when $m = \floor{\frac{k}{k-\ell}}$ and $k-\ell$ divides $k$, $B_{m+1}$ is the $0$-graph and we take $B_{m+1}=\emptyset$ to mean that the empty set is not an edge of~$B_{m+1}$; that is, $B_m$ has at most $r$ edges.

\begin{proofclaim}
If $B_{m+1}\neq \emptyset$, then there is some $t_{m+1}$-set which has degree greater than $r$ in $B_{m}$ and thus $|B_{m}|>r$.  
Observe that for all $j\in [2,m]$,
\begin{align*}
 r |B_j| < \sum_{S \in B_j}\deg_{B_{j-1}}(S) \le \binom{t_{j-1}}{t_j} |B_{j-1}|.
\end{align*}
Hence
\begin{align*}
	|E(H)| = |B_1| > \frac{r}{ \binom{t_{1}}{t_2} }|B_2| > \dots > \frac{r^{m-1}}{ \binom{t_{1}}{t_2} \cdots \binom{t_{m-1}}{t_{m}}}  |B_{m}| > 
	\frac{r^{m}}{ \binom{t_{1}}{t_2} \cdots \binom{t_{m-1}}{t_{m}}}
	=\frac{r^{m}}{g(k,\ell,m)},
\end{align*}
a contradiction.
\end{proofclaim}

We say that a pair of edges $(e,e') \in E(H)\times E(H)$ is \emph{dangerous} if 
\begin{align} \label{eqn:property}
\text{$|e \cap e'| = \ell$  and  $\exists j^* \in [m]$ and $S \in \binom{e \cap e'}{t_{j^*+1}} \setminus B_{j^*+1}$ such that $S \cup (e' \setminus e)\in B_{j^*}$}.
\end{align}
We now prove that every copy of $P^{(k,\ell)}_{\ell+(m+1)(k-\ell)}$ in $H$ contains a dangerous pair of edges.

\begin{claim}\label{clm:path_ell}
Let $P$ be a copy of $P^{(k,\ell)}_{\ell+(m+1)(k-\ell)}$ in $H$ with consecutive edges $e_1, \dots, e_{m+1}$.
Then there exists $i^*\in [m]$ such that $(e_{i^*}, e_{i^*+1})$ is dangerous.
\end{claim}

\begin{proofclaim}
First note that for all $i,i' \in [m+1]$ with $i < i'$, $|e_i \cap e_{i'}| = k - (i'-i)(k- \ell) = t_{i'-i+1}$.  We begin by showing that there exist $i^*,j^* \in [m]$ such that $e_{i^*} \cap e_{i^*+j^*-1}, e_{i^*+1} \cap e_{i^*+j^*} \in B_{j^*}$ and $e_{i^*} \cap e_{i^*+j^*} \notin B_{j^*+1}$.
For all $j \in [m+1]$, let $Q_j$ be the $(t_j,t_{j+1})$-path with consecutive edges $e_1 \cap e_j, \dots, e_{m+2-j} \cap e_{m+1}$.
Note that $Q_1 \subseteq H = B_1$ and $Q_{m+1} \not \subseteq B_{m+1}$ by Claim~\ref{clm:i0} (Again, if $m = \floor{\frac{k}{k-\ell}}$ and $k-\ell$ divides $k$, then $Q_{m+1}$ contains the empty set as an edge but $B_{m+1}$ does not.).  Thus there exists a smallest $j^*\in [m]$ such that $Q_{j^*} \subseteq B_{j^*}$ and $Q_{j^*+1} \not \subseteq B_{j^*+1}$.  Since $Q_{j^*+1} \not \subseteq B_{j^*+1}$, there exists $i^*$ such that $e_{i^*} \cap e_{i^*+j^*} \in Q_{j^*+1} \setminus  B_{j^*+1}$. 

We now show that $(e_{i^*},e_{i^*+1})$ is dangerous. 
Let $S = e_{i^*} \cap e_{i^*+j^*} \subseteq e_{i^*} \cap e_{i^*+1}$, so $S \in \binom{e_{i^*}\cap e_{i^*+1}}{t_{j^*+1}} \setminus B_{j^*+1}$.
Note that $(e_{i^*+1} \setminus e_{i^*}) \cap S = \emptyset$,  $(e_{i^*+1} \setminus e_{i^*})  \subseteq e_{i^*+1} \cap e_{i^*+j^*}$ and 
\begin{align*}
	k-\ell  
	= |e_{i^*+1} \setminus e_{i^*}| 
	\ge | (e_{i^*+1} \setminus e_{i^*}) \cap e_{i^*+j^*}|
	= |e_{i^*+1} \cap e_{i^*+j^*}| - |S|
 = t_{j^*-1} - t_{j^*} = k- \ell.
\end{align*}
Hence $S \cup (e_{i^*+1} \setminus e_{i^*}) = e_{i^*+1} \cap e_{i^*+j^*} \in B_{j^*}$.
Therefore, $(e_{i^*},e_{i^*+1})$ is dangerous. 
\end{proofclaim}

We next show that for every $e\in E(H)$, there are a bounded number of $e'\in E(H)$ such that $(e,e')$ is dangerous.

\begin{claim}\label{clm:danger}
For all $e\in E(H)$, there are at most $r\cdot f(k,\ell,m)$ many $e'\in E(H)$ such that $(e,e')$ is dangerous.
\end{claim}

\begin{proofclaim}
Let $e \in E(H)$.  For all $j \in [m]$, we count the number of~$e' \in E(H)$ such that $(e,e')$ satisfy \eqref{eqn:property} with $j^*=j$.  
There are at most $\binom{k}{\ell}$ choices for~$e \cap e'$ and at most $\binom{\ell}{t_{j+1}}$ choices for $S \in \binom{e \cap e'}{t_{j+1}} \setminus B_{j+1} $.
Since $S \notin B_{j+1}$ and $S \cup (e' \setminus e) \in B_j$, there are at most $r$ choices for~$e' \setminus e$. 
Hence there are at most $r \binom{k}{\ell} \binom{\ell}{t_{j+1}}$ many $e' \in E(H)$ such that $(e,e')$ satisfy \eqref{eqn:property} with $j^*=j$. 
Therefore the total number of~$e' \in E(H)$ such that $(e,e')$ satisfy \eqref{eqn:property} is at most
\begin{align*}
r \binom{k}{\ell} \sum_{j \in [m]} \binom{\ell}{t_{j+1}} = r\cdot f(k,\ell,m) 
\end{align*}
as required. 
\end{proofclaim}

Finally, we use Claim~\ref{clm:danger} to color the edges of~$H$ such that no dangerous pair receives the same color. Since every copy of $P^{(k,\ell)}_{\ell+(m+1)(k-\ell)}$ contains a dangerous pair of edges by Claim~\ref{clm:path_ell}, this implies that such a coloring of $H$ does not contain a monochromatic~$P^{(k,\ell)}_{\ell+(m+1)(k-\ell)}$.

Let $D$ be the auxiliary digraph such that $V(D) = E(H)$ and $(e,e') \in E(D)$ if and only if $(e,e')$ is dangerous.
Let $G$ be the underlying undirected graph of~$D$.  A proper vertex-coloring of $G$ gives an edge-coloring of~$H$ (with the same set of colors) without a monochromatic~$P^{(k,\ell)}_{\ell+(m+1)(k-\ell)}$.
Note that for all $U \subseteq V(D)$, $G[U]$ has at most $ \Delta^+(D) |U|$ edges and so $\delta(G[U]) \le 2 \Delta^+(D)$.  This implies $\chi(G) \le 2\Delta^+(D)+1\le 2r\cdot f(k,\ell,m)+1$ where the last inequality follows from Claim~\ref{clm:danger}. 
\end{proof}

We now prove Theorem~\ref{thm:main2ell}.

\begin{proof}[Proof of Theorem~\ref{thm:main2ell}]
By Lemma~\ref{lem:monoton_ell} we have $\hat{R}_r(P_{\ell+(m+1)(k-\ell)}^{(k, \ell)})\leq (2emr)^m$.
By Lemma~\ref{lemma:p2k_ell_d} with $\floor{\frac{r-1}{2 f(k,\ell,m)}}$ in place of $r$ (noting that $2\floor{\frac{r-1}{2f(k,\ell,m)}}f(k,\ell,m)+1\leq r$), we have 
\begin{equation*}
\hat{R}_r(P_{\ell+(m+1)(k-\ell)}^{(k, \ell)})
>\frac{\left(\floor{\frac{r-1}{2f(k,\ell,m)}} \right)^m}{g(k,\ell,m)}
= \Omega_{k,\ell,m}(r^{m}). \qedhere
\end{equation*}
\end{proof}

\section{Corollaries, more precise bounds, and a further extension}\label{sec:p4}

First we prove Corollary~\ref{cor:ell}, then we obtain the more precise bounds given in Theorem~\ref{thm:P_{k+2}}, Corollary~\ref{cor:k3}, and Theorem~\ref{thm:loose-lb}.  At the end, we mention one more extension of our results. 

\begin{proof}[Proof of Corollary~\ref{cor:ell}]
By Theorem~\ref{thm:main2ell} (with $k-1$, $\ell-1$, $\floor{\frac{k}{k-\ell}}-1$ in place of $k,\ell,m$ respectively -- noting that $\floor{\frac{k}{k-\ell}}-1\leq \floor{\frac{k-1}{k-\ell}}$), we have $\hat{R}_r(P^{(k-1,\ell-1)}_{\ell-1+\floor{\frac{k}{k-\ell}}(k-\ell)})=\Omega_{k}(r^{\floor{\frac{k}{k-\ell}}-1})$.

Thus by Theorem~\ref{thm:main1ell} and Proposition~\ref{prop:ramlower_ell}, we have
\[
	\hat{R}_r(P_{n}^{(k,\ell)})
	\geq \frac{1}{k}\hat{R}_r(P_{\ell-1+\floor{\frac{k}{k-\ell}}(k-\ell)}^{(k-1,\ell-1)}) \left( 1+ \frac{r-1}{{(k-\ell)\ceiling{\frac{k}{k-\ell}}}}\right) (1-o(1))n
	= \Omega_k(r^{\floor{\frac{k}{k-\ell}}} n).\qedhere
\]
\end{proof}

To give an improved lower bound on $\hat{R}(P_4)$, we first need a few definitions.  The \emph{arboricity} of a graph~$G$, denoted by~$\mathrm{arb}(G)$, is the smallest number of forests needed to decompose the edge set of~$G$.  The \emph{star arboricity} of a graph~$G$, denoted by~$\mathrm{arb}_{\star}(G)$, is the smallest number of star-forests needed to decompose the edge set of $G$.  Note that since every forest can be decomposed into at most two star-forests we have that for all $G$, $\mathrm{arb}_{\star}(G)\leq 2\mathrm{arb}(G)$.  A well-known result of Nash-Williams~\cite{NW1,NW2} says that $\mathrm{arb}(G)\leq k$ if and only if $|E(H)|\leq k(|V(H)|-1)$ for all subgraphs $H\subseteq G$.

\begin{proof}[Proof of Theorem~\ref{thm:P_{k+2}}]
Note that \eqref{eq:trivial} and \eqref{thm:bier} imply that 
\begin{align*}
	\hat{R}_r(P_{4}) \le \binom{R_r(P_{4})}{2}\leq \binom{2r+2}{2}=(r+1)(2r+1).
\end{align*}

It remains to show that $\hat{R}_r(P_4) > \frac{r^2}{2}$.
Let $G$ be a graph with $|E(G)|\leq \frac{r^2}{2}$.  
Note that a star-forest is $P_4$-free, so it suffices to show that $\mathrm{arb}_{\star}(G) \le r$.
If $H\subseteq G$ with $|V(H)|\leq r$, then $$|E(H)|\leq \frac{|V(H)|(|V(H)|-1)}{2}\leq \frac{r}{2}(|V(H)|-1)$$ and if $|V(H)|\geq r+1$, then $$|E(H)|\leq |E(G)|\leq \frac{r^2}{2}\leq \frac{r}{2}(|V(H)|-1).$$  Now Nash-Williams' theorem mentioned above implies that $\mathrm{arb}_{\star}(G)\leq 2\mathrm{arb}(G)\leq 2\frac{r}{2}=r$, and thus $G$ has an $r$-coloring with no monochromatic $P_4$.
\end{proof}

\begin{proof}[Proof of Corollary~\ref{cor:k3}]
Theorem~\ref{thm:main1ell} and Proposition~\ref{prop:ramlower_ell} (with $k=3$ and $\ell=2$) imply that 
\begin{align*}
	\hat{R}_r(P^{(3)}_{n})&\geq \frac{1}3 \hat{R}_r(P_{4}) R_r(P^{(3)}_{n-18\hat{R}_r(P_{4})})\\
	&>  \frac{1}3 \hat{R}_r(P_{4}) \left(\frac{r+2}{3}\left( n - 18 \hat{R}_r(P_4) \right) -2(r-1)\right)
	\ge 
	\begin{cases}
	\frac{28 n }{9} - 390 & \text{if  $r =2$}\\
	\frac{r^2(r+2)}{12}n-O(r^5) & \text{if  $r \ge 3$}
\end{cases},
\end{align*}
where the last inequality holds by~\eqref{P4} and Theorem~\ref{thm:P_{k+2}} respectively. 
\end{proof}

To prove Theorem~\ref{thm:loose-lb}, we use the following result of Alon, Ding, Oporowski, and Vertigan~\cite{ADOV}.
\begin{theorem}[{\cite[Theorem~4.1]{ADOV}}] 
\label{thm:ADOV}
Every graph with maximum degree $\Delta$ can be vertex $\ceiling{\frac{\Delta + 2}{3}}$-colored such that every connected subgraph in every color class has at most $12\Delta^2$ vertices.
\end{theorem}

\begin{proof}
[Proof of Theorem~\ref{thm:loose-lb}]
Let $n>12r^2(k-\ell)+\ell$.  We will show that $\hat{R}_r(P_{n}^{(k,\ell)})\ge \frac{r}{k^2}R_{ \floor{\frac{2r-2}{3}}}(P_{n}^{(k,\ell)})$, from which the result will follow.  

Let $H=(V,E)$ be a $k$-graph with \[|E| < \frac{r}{k^2}R_{\floor{\frac{2r-2}{3}}}(P_{n}^{(k,\ell)}).\] 
Let $S=\{v\in V: d(v)<\frac{r}{k}\}$.
So we have \[|V\setminus S|\frac rk  \le\sum_{v\in V(H)} d(v) = k|E|  < \frac{r}{k}R_{\floor{\frac{2r-2}{3}}}(P_{n}^{(k,\ell)}),\]
and thus $|V\setminus S| \le R_{\floor{\frac{2r-2}{3}}}(P_{n}^{(k,\ell)})$. 
Thus we may color $H \setminus S$ with $\floor{\frac{2r-2}{3}}$ colors without creating a monochromatic $P_n^{(k,\ell)}$.

Now let $E_S$ be the set of edges in $H$ that are incident to~$S$.
We will use the remaining $r-\floor{\frac{2r-2}{3}}=\ceiling{\frac{r+2}{3}}$ colors on the edges incident to $E_S$. 
Form an auxiliary graph~$G$ whose vertex set is $E_S$ and $e_1, e_2$ are adjacent in $G$ if and only if $|e_1\cap e_2| = \ell$. 
By the definition of $S$ we have $\Delta(G) < k\cdot \frac rk = r$.
By Theorem~\ref{thm:ADOV}, we may use $\ceiling{\frac{r + 2}{3}}$ new colors to color the vertices of~$G$ such that every connected subgraph in every color class has at most $12\Delta^2$ vertices.  Thus applying these colors to the corresponding edges in~$E_S$, we obtain a coloring of $H$ with no monochromatic $P_n^{(k,\ell)}$ for all $n>12r^2(k-\ell)+\ell$.

Finally, by Proposition \ref{prop:ramlower_ell}, we have $\hat{R}_r(P_{n}^{(k,\ell)})\ge \frac{r}{k^2}R_{ \floor{\frac{2r-2}{3}}}(P_{n}^{(k,\ell)})=\Omega_k(r^2 n).$
\end{proof}

We conclude this section with one more extension of our results.  Our proof of Theorem~\ref{thm:main1ell} can also be used to prove the following theorem.

\begin{theorem}\label{thm:link}
Let $k,r\geq 2$ be integers, let $\cH$ be a family of $k$-graphs, and let $F$ be a $(k-1)$-graph such that for all $H\in \cH$ the link graph of every vertex in $H$ contains a copy of~$F$.  Then 
$$\hat{R}_r(\cH)\geq \frac{1}{k}\hat{R}_r(F)\cdot R_r(\cH).$$
\end{theorem}

\begin{proof}
This is just a generalization of Theorem~\ref{thm:main1ell}, but the proof is easier since we are assuming that every vertex in $H$ contains a copy of $F$ in its link graph.  Indeed, if we color the edges as in the proof of Theorem~\ref{thm:main1ell}, there can be no monochromatic copy of any $H\in \cH$ which intersects $S$ (otherwise, if $i\in [c]$ is minimum such that $V(H)\cap S_i\neq \emptyset$, then $v\in V(H)\cap S_i$ has no monochromatic copy of $F$ in its link graph). 
\end{proof}

As an application, let $\mathcal{C}^{(k)}_{\geq 2k-1}$ denote the family of $k$-uniform tight cycles on at least $2k-1$ vertices.  Note that for all $C\in \mathcal{C}^{(k)}_{\geq 2k-1}$ and all $v\in V(C)$, the link graph of $v$ in $C$ contains a copy of $P_{2k-2}^{(k-1)}$.  We obtain the following corollary of Theorem~\ref{thm:link}, Theorem~\ref{thm:main2ell}, and Proposition~\ref{prop:manycolour}.

\begin{corollary}\label{cor:cycles}
For all integers $k\geq 2$ and infinitely many integers $r\geq 2$, we have $\hat{R}_r(\mathcal{C}^{(k)}_{\geq 2k-1})=\Omega_k(r^k)$. 
\end{corollary}

The significance of this corollary is that $\hat{R}_r(\mathcal{C}^{(k)}_{\geq 2k-1})=\Omega_k(r^k)$ whereas Theorem~\ref{thm:main2ell} implies that $\hat{R}_r(P^{(k)}_{2k-1})=\Theta_k(r^{k-1})$.

\begin{proof}
By Theorem~\ref{thm:link} and Theorem~\ref{thm:main2ell} we have
$$\hat{R}_r(\mathcal{C}^{(k)}_{\geq 2k-1})\geq \frac{1}{k}\hat{R}_r(P_{2k-2}^{(k-1)})\cdot R_r(\mathcal{C}^{(k)}_{\geq 2k-1}) \geq  \Omega_{k}(r^{k-1})\cdot R_r(\mathcal{C}^{(k)}_{\geq 2k-1}).$$
Since $R_r(\mathcal{C}^{(k)}_{\geq 2k-1})\geq R_r(P^{(k)}_{2k-1})$ and Proposition~\ref{prop:manycolour} implies that for infinitely many $r$, $R_r(P^{(k)}_{2k-1})\geq (1-o(1))(k-1)r$, we have that for infinitely many $r$, $\hat{R}_r(\mathcal{C}^{(k)}_{\geq 2k-1})=\Omega_k(r^k)$. 
\end{proof}

\section{Conclusion}\label{sec:conclusion}

In the proof of Theorem~\ref{thm:main2ell} (for simplicity, restricted to the case $\ell=k-1$ and using estimates $f(k,k-1,m)\leq k(ek/m)^m$ and $g(k,k-1,m)\leq k^m$) we show that for all $1\leq m\leq k$, 
\begin{equation}\label{eq:estimate}
\left(\frac{1}{2k^2(ek/m)^m}\right)^mr^m\leq  \hat{R}_r(P^{(k)}_{k+m})\leq \hat{R}_r(P^{(m)}_{2m})\leq (2em)^m r^m.
\end{equation}
It would be nice to improve these bounds; in particular, the lower bound.  A related question is as follows.

\begin{problem} \label{prob:k+p}
Is it true that for all positive integers $m$ there exists constants $c_m, C_m$ such that for all $k,r\geq 2$, $c_mr^m\leq \hat{R}_r(P^{(k,\ell)}_{\ell+(m+1)(k-\ell)})\leq C_mr^m$.  
\end{problem}

One approach to Problem~\ref{prob:k+p} and improving the bounds in \eqref{eq:estimate} is as follows.   Note that Lemma~\ref{lem:monoton_ell} implies that $\hat{R}_r(P^{(k,\ell)}_{\ell+(m+1)(k-\ell)}) \leq \hat{R}_r(P^{(m)}_{2m})$ for all $1 \le m \leq \floor{\frac{k}{k-\ell}}$.  At the moment we have no evidence to refute the possibility that these are equal.  So formally, we raise the following problem.  In either case, if the answer is yes, then the answer is yes in Problem~\ref{prob:k+p} and improves the lower bound in \eqref{eq:estimate}.

\begin{problem}\label{prob:2k2k+1}
Is it true that $\hat{R}_r(P^{(k,\ell)}_{\ell+(m+1)(k-\ell)}) = \hat{R}_r(P^{(m)}_{2m})$ for all integers $1 \le m \leq \floor{\frac{k}{k-\ell}}$?  If not, is it true that $\hat{R}_r(P^{(k,\ell)}_{\ell+(m+1)(k-\ell)})=\Omega_m(\hat{R}_r(P^{(m)}_{2m}))$? 
\end{problem}

In Theorem~\ref{thm:loose-lb} we are able to improve the lower bound coming from Corollary~\ref{cor:ell} by an extra factor of $r$. It would be interesting to do this for all $k/2<\ell\leq k-1$ where $k-\ell$ does not divide $k$.  In Theorem~\ref{thm:main1ell}, we use the fact that for every vertex $v$ in the interior of a $(k,\ell)$-path, the link graph of $v$ is a $(k-1,\ell-1)$-path with at least $\floor{\frac{k}{k-\ell}}$ edges, but perhaps there is some alternate proof which is able to make use of the fact that every edge of the path contains some vertex whose link graph is a $(k-1,\ell-1)$-path with $\ceiling{\frac{k}{k-\ell}}$ edges.  If so, this would imply $\hat{R}_r(P^{(k, \ell)}_{n}) = \Omega_k(r^{\ceiling{\frac{k}{k-\ell}}}n)$ (see the proof of Corollary~\ref{cor:ell} and note that $\ceiling{\frac{k}{k-\ell}}-1\leq \floor{\frac{k-1}{k-\ell}}$).

\begin{problem}
In Theorem~\ref{thm:main1ell} can we replace 
``$q = \ell-1+\floor{\frac{k}{k-\ell}}(k-\ell)$'' with the term ``$q = \ell-1+\ceiling{\frac{k}{k-\ell}}(k-\ell)$''?
\end{problem}

Regarding Corollary~\ref{cor:cycles}, it would be interesting to know if $R_r(\mathcal{C}^{(k)}_{\geq 2k-1})=O_k(r)$.  If so, this would imply by \eqref{eq:trivial} and Corollary~\ref{cor:cycles} that for infinitely many $r$, $R_r(\mathcal{C}^{(k)}_{\geq 2k-1})=\Theta_k(r)$.  

\begin{problem}\label{prob:cycles}
Is it true that for all integers $k,r\geq 2$, $R_r(\mathcal{C}^{(k)}_{\geq 2k-1})=O_k(r)$?  
\end{problem}

One thing to note about Problem~\ref{prob:cycles} is that unlike in Corollary~\ref{cor:upper} (where we show that $R_r(P_{2k-1}^{(k)})\leq rk(2k-1)$), Problem~\ref{prob:cycles} cannot be solved by applying a Tur{\'a}n-type result to the majority color class.  B.~Janzer~\cite{J} proved that there exists $k$-uniform hypergraphs on $N$ vertices with $\Omega(N^{k-1}\frac{\log N}{\log \log N})$ edges which have no tight cycles at all.  On the other hand, a result of Letzter~\cite{L} says that every $k$-uniform hypergraph on $N$ vertices with no tight cycles has $O((\log N)^5N^{k-1})$ edges.\footnote{Letzter's proof gives no explicit lower bound on the length of the tight cycle one obtains beyond this threshold; however, the proof in \cite[Theorem 6]{L} can be slightly modified to give a tight cycle with at least $3k-2\geq 2k-1$ vertices.  We thank Shoham Letzter for pointing this out to us.}  As in Corollary~\ref{cor:upper}, this implies that $R_r(\mathcal{C}^{(k)}_{\geq 2k-1})=O_k(r(\log r)^{5/k})$.

Finally, our lower bounds on $\hat{R}_r(P_{n}^{(k)})$ are for fixed $k$ and growing $r$.  Can we say anything about the case of fixed $r$ and growing $k$ (like Winter did for $r=2$)?

\begin{problem}\label{prob:main}
What is the growth rate of $\hat{R}_r(P^{(k)}_{2k})$ for fixed $r$ and growing $k$?  In particular, when $r=2$, is $\hat{R}(P^{(k)}_{2k})=\omega(k\log_2k)$?  If so, then Theorem~\ref{thm:main1ell} would give an improvement over Winter's result.
\end{problem}

\tbf{Acknowledgements:}  We thank the referees for the their helpful comments which allowed us to greatly improve the presentation of the paper.

\bibliographystyle{abbrv}
\bibliography{references}

\end{document}